\documentclass{amsart}

\usepackage[T1]{fontenc}

\theoremstyle{plain}
 \newtheorem{thm}{Theorem}[section]
 \newtheorem{cor}[thm]{Corollary}
 \newtheorem{lem}[thm]{Lemma}
 \newtheorem{prop}[thm]{Proposition}
 \theoremstyle{definition}
 \newtheorem{defn}[thm]{Definition}
 \theoremstyle{remark}
 \newtheorem{rem}[thm]{Remark}
 \newtheorem{ex}{Example}
 \numberwithin{equation}{section}

\newcommand{\BB}{{\mathbf  B}}
\newcommand{\TT}{{\mathbb T}}
\newcommand{\CC}{{\mathbb C}}
\newcommand{\NN}{{\mathbb N}}
\newcommand{\NNN}{{\mathcal N}}

\newcommand{\DD}{{\mathbb D}}

\newcommand{\CCC}{\mathcal C}

\newcommand{\FF}{\mathcal F}
\newcommand{\FFF}{\mathbb F}
\newcommand{\GG}{\mathcal G}
\newcommand{\GGG}{\mathbb G}

\newcommand{\LL}{\mathcal{L}}
\newcommand{\MM}{\mathcal{M}}
\newcommand{\HH}{{\mathcal H}}

\renewcommand\Re{\, {\text{Re}\,}}
\renewcommand\Im{\, {\text{Im}\,}}

\newcommand{\la}{{\lambda}}

\newcommand{\phii}{{\varphi}}
\def\inpi{\frac{1}{2\pi}\int_0^{2\pi}}
\renewcommand{\span}[1]{\overline{\mathop{\rm span}(#1  )}}
\newcommand{\rank}{\mathop{\rm rank}\nolimits}
\def\<{\langle}
\def\>{\rangle}
\def\tn{\;|\!|\!|\,}

\begin{document}

\title[Kernels of vector-valued Toeplitz operators]{Kernels of vector-valued Toeplitz operators}
\author[Chevrot]{Nicolas Chevrot}
\address{D\'epartement de math\'ematiques et de statistique\\
Universit\'e Laval\\
Qu\'ebec (QC),\\
Canada\\
G1V 0A6 }
\email{nicolas.chevrot.1@ulaval.ca}

\date{june 2009}

\thanks{ The author is very grateful to Pr. Thomas Ransford for his advice and to Pr.  Andreas Hartmann.}
\keywords{Toeplitz operators, de Branges--Rovnyak spaces, vector-valued functions}

\subjclass{Primary: 47B32, 30D55 Secondary: 46C07, 46E40,47B35}

\begin{abstract} 
Let $S$ be the shift operator on the Hardy space $H^2$ and let $S^*$ be its adjoint.
A closed subspace $\FF$ of $H^2$ is said to be nearly $S^*$-invariant if  every element $f\in\FF$ with $f(0)=0$ satisfies $S^*f\in\FF$. In particular, the kernels of Toeplitz operators are nearly $S^*$-invariant subspaces. Hitt  gave the description of these subspaces. They are of the form $\FF=g (H^2\ominus u H^2)$ with  $g\in H^2$ and $u$ inner, $u(0)=0$.
 A very particular fact is that the operator of multiplication by $g$  acts as an isometry on $H^2\ominus uH^2$. Sarason  obtained a characterization of the functions $g$ which act isometrically on $H^2\ominus uH^2$. 
 Hayashi  obtained the link between the symbol $\phii$ of a Toeplitz operator and the functions $g$ and $u$ to ensure that a given subspace $\FF=gK_u$ is the kernel of $T_\phii$.
Chalendar, Chevrot and Partington studied the  nearly $S^*$-invariant subspaces for vector-valued functions. In this paper, we investigate   the generalization of Sarason's  and Hayashi's results in the vector-valued context. 
\end{abstract}

\maketitle

\section{Introduction}

To begin this section, we present  the scalar results of Hitt, Sarason and Hayashi which will be generalized throughout  this paper.  

We denote by $H^2$ the classical Hardy space of analytic functions on the unit disc $\DD$, and by $H^2(\CC^m)$ the $\CC^m$-vector-valued Hardy space consisting of $m$ copies of $H^2$. The shift $S$ is the  operator of multiplication by the variable $z$ and $S^*$ is its adjoint.  The (closed) $S^*$-invariant subspaces of $H^2$ are called model subspaces. They are of the form $K_u=H^2\ominus uH^2$, where $u$ is an inner function. 

 For $\phii\in L^\infty$, the Toeplitz operator  with symbol $\phii$ is defined  by $T_\phii f:=p_+(\phii f)$, where $p_+$ is the orthogonal projection from $L^2$ onto $H^2$.

Hitt \cite{Hitt} introduced the \textit{nearly $S^*$-invariant subspaces}:

\begin{defn}
A closed subspace $\FF$ of $H^2$ is said to be a \textit{nearly $S^*$-invariant subspace} if every element $f\in\FF$ with $f(0)=0$ satisfies $S^*f\in\FF$.
\end{defn}
 In  particular, the kernel of a Toeplitz operator is a nearly $S^*$-invariant subspace.
Hitt obtained the complete description of this spaces:
\begin{thm}[Hitt,  1988]\label{hitt}
Let $\FF$ be a non-trivial nearly $S^*$-invariant subspace.  Let $g$ be the unique unit-norm function in $\FF$, positive at the origin, that is orthogonal to $\FF\cap zH^2$. Then there exists an inner function $u$ vanishing at zero such that,  for all $f\in \FF$, there exists a unique $f_0 \in K_u$ and $f=gf_0$. Furthermore, $\|f\|_2=\|f_0\|_2$. In other words, multiplication by $g$ acts isometrically on $K_u$.
\end{thm}

   Two questions arise.    
\begin{enumerate}
	\item The first one was already posed by Sarason in \cite{sarason:art}  where he made this remark:   "{\itshape The latter theorem leaves  mysterious the relation between the function $g$ and the space $K_u$.  Given a function $g$ of unit norm in $H^2$, what are the $S^*$-invariant subspaces $K_u$ that can arise with $g$ in Hitt's theorem? }"
  \item Which nearly $S^*$-invariant subspaces are kernels of Toeplitz operators. 
\end{enumerate}

 Sarason obtained the following answer to the first question:
 
\begin{thm}[Sarason, 1988]\label{sarason}
Let $g$ be an outer function of unit norm, and $u$ an inner function with $u(0)=0$.
We define two analytic functions on the disc:
\[
f(z):=\dfrac{1}{2\pi} \int_0^{2\pi} \dfrac{e^{i\theta}+z}{e^{i\theta}-z} |g(e^{i\theta})|^2 \ d\theta\ \mbox{ and }\ b(z):=\dfrac{f(z)-1}{f(z)+1}.
\]
Then the following statements are equivalent:
\begin{enumerate}
	\item multiplication by $g$ acts isometrically from $K_u$ to $\FF$;
	\item $bH^2\subset uH^2$ (\textit{i.e.} $b=ub_0$);
	\item $K_u \subset (1-T_bT_{\bar{b}})^{1/2}H^2.$ 
\end{enumerate}
\end{thm}

The answer to the second question is given by Hayashi in \cite{hayashi,hayashi85,hayashi90} and Sarason found an alternative proof in \cite{sarasontoeplitz}. This answer is expressed in terms of exposed points of the unit ball of $H^1$, also called \textit{rigid functions}. 

Before stating Hayashi's result, we need some definitions. With the previous notation, let $\FF=g K_u$ be a nearly $S^*$-invariant space and let $b$ be the function associated to $g$ as in Theorem~\ref{sarason}. Because $\log (1-|b|^2)$ is integrable, we can build an outer function $a$ such that $|a|^2+|b|^2=1$ a.e.\ on $\TT$. Then $(b,a)$ is called a \textit{corona pair (or pair)} associated to $g$. Thanks to Theorem~\ref{sarason}, $b=ub_0$. If $\FF$ is the kernel of a Toeplitz operator, then   $(b_0,a)$ is a  corona pair associated to the outer function $g_0:=a/(1-b_0)$. Some pairs, called \textit{special pairs},  verify an additional property which will be precisely defined   in section~\ref{matricialhayashi}.  Admitting this, we can reformulate Hayashi's result as follows (see also \cite{sarasontoeplitz}):

\begin{thm}[Hayashi, 1985]\label{hayashi}
The subspace $\FF=g K_u$ is the kernel of a Toeplitz operator if and only if the pair $(b_0,a)$ is special and $g_0^2$ is rigid.
\end{thm}

We would like to generalize the previous theorems to vector-valued functions. The paper is organized as follows. In section~\ref{Hardy}, we define the vector- or matrix-valued objects: we recall the inner-outer matricial factorization, we comment on the generalization of Theorem~\ref{hitt}, and we  recall the  definition of  de Branges--Rovnyak spaces, the vector-valued analogue of $\HH(b):=(1-T_bT_{\bar{b}})^{1/2}H^2$  appearing in Theorem \ref{sarason}.   

In  section~\ref{matricialsarason}, we transcribe Sarason's approach to the vectorial case. We build the analogue of the functions $b$ and $u$.  Thanks to de Branges--Rovnyak spaces, we obtain the matricial version of Theorem~\ref{sarason}.
    The matrices do not commute, so we need to modify the original scalar proof given by Sarason. An example illustrates this kind of problem.

In section~\ref{kernel}, we would like to describe the kernels of Toeplitz operators.  We begin with some examples. This allows us to illustrate the difficulties due to the dimension, and to establish some notation.   We then investigate the descriptions of kernels of Toeplitz operators of finite dimension. 

Finally, in section~\ref{matricialhayashi}, we obtain the full description of the kernels of Toeplitz operators. We establish the desired generalization of Hayashi's Theorem.

\section{Hardy spaces of vector-valued functions}\label{Hardy}

\subsection{Inner-outer factorization}\label{innerouter}

As usual with Hardy spaces, we identify a function with its radial limits. 

Let $\FFF,\GGG$ be two subspaces of $\CC^m$ of   dimension $r$. Nikolskii, in \cite{Nik1} page 14, calls  $\Theta\in H^\infty(\FFF\to \GGG)$ an inner function  if its boundary values $\Theta(\xi)$ are surjective isometries for a.e.\ $\xi \in \TT$. 

It will be more convenient to say that $\Theta\in H^\infty(\CC^m\to \CC^m)$ is an \textit{inner function}  if its boundary values $\Theta(\xi)$ are partial isometries for a.e.\ $\xi \in \TT$, with   kernel and range independent of $\xi$ a.e.\ in $\TT$.  In other words, an inner function is a square-matrix-valued function such that
 there exist two subspaces  $\FFF,\GGG$  of $\CC^m$ with the same dimension $r$ for which $\Theta|_\FFF \in H^\infty(\FFF\to \GGG)$ is an inner function in the sense of Nikolskii.  The rank of $\Theta(\xi)$ is  equal to $r$ for a.e.\ $\xi \in\TT$.

Here are two examples of inner functions of rank 2. The first one will be discussed later (see Theorem~\ref{garcia}). Let $\theta$ be an inner scalar function and $a,b\in K_{z\theta}$ verifying $|a|^2+|b|^2=1$ a.e.\ on $\TT$. Define $\phii\in H^\infty(\CC^2\to\CC^2)$ and $\Theta \in H^\infty(\CC^3\to\CC^3)$  by the following formulae:
\[\FFF= \begin{pmatrix}\CC \\  \CC \\0 \end{pmatrix}, \ \ \GGG:=   \begin{pmatrix}\CC \\  0 \\ \CC \end{pmatrix}, \ \
\phii:=
\begin{pmatrix}
a & -b \\
\theta \bar{b} & \theta \bar{a}
\end{pmatrix}
\ \mbox{ and }\ \Theta:=
\begin{pmatrix}
a &0 & -b  \\
\theta \bar{b}&0 & \theta \bar{a} \\
0& 0& 0
\end{pmatrix}.
\]
 Both $\phii,\Theta$ are inner of rank 2. Note that $\phii\in H^\infty(\CC^m\to \CC^m)$ is inner of rank $m$ if and only if $\det \phii$ is inner.

 Recall the Beurling--Lax Theorem \cite{Lax}~:
If a closed subspace $\MM\subset H^2(\CC^m)$ is invariant by the shift, then there exists an inner function $\Theta$ such that $\MM=\Theta H^2(\CC^m)$. This description is unique up to multiplication by an unitary matrix. 

Next, we recall the notion of outer vector-valued function. The outer scalar functions are  cyclic vectors for the shift. For $g\in H^2(\CC^m)$, we define $\GG$, the smallest $S-$invariant subspace containing $g$, by $\GG:=\span{S^k g : k\in \NN}$. Thanks to Beurling--Lax theorem, there exists $\Theta$, inner with of 1,  such that $\GG=\Theta H^2(\CC^m)$. We  say that $g$ is \textit{outer} if $\Theta$ is a constant matrix. Then $\GG=H^2(\Theta(0)\CC^m)$. It will be useful to write $\GGG:=\Theta(0)\CC^m$. Finally, the function $g$ is a cyclic vector for $S$ in $H^2(\GGG).$   

We extend this  construction to define the outer matrix-valued functions.
Let $g_1,\dots ,g_r\in H^2(\CC^m)$, with $r\leq m$, be  a  independent family of vector-valued functions. Let $G\in H^2(\CC^r\to \CC^m)$ be the rectangular matrix-valued functions where the columns are $(g_\ell)_{\ell\leq r}$. In this case we write $G=[g_1,\dots ,g_r]$. 
  It is said to be \textit{outer} if $\GG:=\span{S^k g_\ell : \ell\leq r, k\in \NN}=\Theta H^2(\CC^m)$, where $\Theta$ is a constant partial isometry of rank $r$. Then, we will write $\GGG=\Theta(0)\CC^m $, $\dim \GGG=r$, and  $\GG=H^2(\GGG)$. Due to the rank theorem, there exists an unitary mapping  $\Theta_0:\CC^r\to\GGG$.  To $G$, we associate $\tilde{G}\in H^2(\CC^r\to \CC^r)$   such that $ G:=\Theta_0 \tilde{G}$.   This allows us to translate the properties of square-matrix-valued functions to rectangular ones.

  For more details about inner-outer factorization of square matrix-valued functions with determinant different from zero, see \cite{katsnelson}. In particular the Definition 5.3 in \cite{katsnelson} of \textit{Beurling left outer function} coincides with that of outer given above. The Smirnov--Nevanlinna class $\NNN^+( \CC^m\to \CC^m)$ of square matrix-valued functions is the set of all matrices with entries in the scalar Smirnov--Nevanlinna class.  The Definition 3.1 in \cite{katsnelson} of outer function  in $\NNN^+( \CC^m\to \CC^m)$ is that $E$ is outer if $\det \ E$ is outer in $\NNN^+ $. The authors shows that all  definitions of outer functions are equivalent in $H^2(\CC^m\to \CC^m)$.  
  Theorem 5.4 in \cite{katsnelson} says that, given a function $F$ in $\NNN^+( \CC^m\to \CC^m)$, $\det F(z) \not\equiv 0$, there exist functions $F_i$ inner and $F_o$ outer (resp. $F_i',F_o'$) , unique up to a unitary matrix, such that  $F=F_i F_o$  (resp. $F= F_o'F_i'$).  Furthermore,  Theorem 3.1 of \cite{katsnelson}  will be useful later:    
  Let $E\in \NNN^+(\CC^m\to \CC^m)$ an outer square-matrix-valued function. Then $\det(z)\not=0$ for all $z\in\DD$ and $E^{-1}\in \NNN^+(\CC^m\to \CC^m)$.

 \subsection{Nearly $S^*$-invariant subspaces of $H^2(\CC^m)$}
 The next result is the description of the nearly $S^*$-invariant subspaces of $H^2(\CC^m)$. For more details, see \cite{nearly}.
\begin{thm}\label{descriptionnearly}
Let $\FF\subset H^2(\CC^m)$ be a non-trivial nearly $S^*$-invariant subspace. Let $(g_1,\dots, g_r)$ be a orthonormal basis of 
	\[ 
	W:= \FF\cap \left\{ \FF  \cap zH^2(\CC^m) \right\}^\perp.
	\]
Then   $r:=\dim\, W\leq m$ and there exist   an integer $r'$, $1\leq r'\leq r$, and $U \in H^\infty(\CC^r\to \CC^r) $ inner, $\rank \, U=r' $, such that 
\[
\FF=\left[g_1,\dots , g_r \right]\left( H^2(\CC^r)\ominus U H^2(\CC^r) \right)=G K_U.
\]
  For all $f\in\FF$, there exists an unique $f_0\in K_U$ such that $f=Gf_0$. Furthermore, $ \left\| f_0  \right\|_{H^2(\CC^r)} =\| f \|_{H^2(\CC^m)}.$
\end{thm}

  Because the columns of $G$ form an orthonormal basis of $W$, the norm of $G\in H^2(\CC^r\to\CC^m)$ is 1. For any  $h\in H^2(\CC^r)$, we define $T_G h$ to be  the Fourier projection  of the $L^1(\CC^m)$ function $Gh$ on $H^2(\CC^m)$. It is an unbounded operator, but, as in the scalar case, it is an isometry on $K_U$.
   

\subsection{De Branges--Rovnyak spaces}

Now, we will recall the definition and the main properties of de Branges--Rovnyak spaces. For more details, see the first chapter of \cite{sarason:book}. Let $H_1$ and $H$ be two Hilbert spaces and $\BB\in \LL(H_1,H)$ be a bounded operator. We define $\MM(\BB)$ to be the range space $\BB H_1$ with the inner product that makes $\BB$ be a coisometry on $H$:  
\[
\forall f,g\in H_1\cap (\ker  \BB)^\perp, \ \  \< Bf,Bg\>_{\MM(\BB)}:=\<f,g\>_{H_1}.
\]
For a contraction $\BB$, the inclusion is a contraction from $\MM(\BB)$ to $H.$
The complementary space $\HH(\BB)$ is defined to be $\MM\left((Id_H-\BB \BB^*)^{1/2}\right)$. In the particular case where $\BB$ is the multiplication by an inner function $B$, then $\MM(\BB)=BH^2(\CC^m)$ and $\HH(\BB)=K_B$. In this case, the inner products of $\MM(\BB)$ and  $\HH(\BB)$  coincide with the $H^2$ inner product and  these two spaces are really complementary spaces in the $H^2$ sense. In this article, $H$ and $H_1$ will be Hardy spaces like $H^2(\CC^m)$  or   closed subspaces of $H^2(\CC^m)$ isometrically equivalent to $H^2(\CC^r)$, and $\BB$ will be the multiplication by a matrix $B$ in the unit ball of $H^\infty(\CC^r\to \CC^m)$. 

The reproducing kernels in  $H^2(\CC^m)$ are  $k_{\la}u:=\frac{1}{1-\bar{\la}z}u$ for $\la\in\DD$ and $u\in \CC^m$. Thus, for all $f\in H^2(\CC^m)$, the reproducing kernels verify  \[ \<f,k_{\la}u\>_2=\<f(\la),u\>_{\CC^m}.\] Because the inclusion from $\HH(B)$ to $H^2$ is contractive, de Branges--Rovnyak spaces have kernel functions, and a simple calculation shows that 
\[
k_{\la}^B u:=\dfrac{Id_r-B(z)B(\la)^*}{1-\bar{\la}z}u \ \ \mbox{ and }\ \<f,k_\la^B u\>_{\HH(B)}=\<f(\la),u\>_{\CC^m}.
\]
Given a symbol $B$, we write $\MM(B)$ (resp. $\HH(B)$) instead of $\MM(T_B)$ (resp. $\HH(T_B)$).

\section{Toeplitz operators acting as an isometry on a model space}\label{matricialsarason}
In this section, we verify that the tools used by Sarason \cite{sarason:art} can be applied to matrix-valued functions. 

\subsection{A matricial intertwining}

Let  $(g_\ell)_{\ell\leq r}$ be an orthogonal basis of $W$ and let $G\in H^2(\CC^r\to \CC^m)$ be the matrix-valued function $[g_1,\dots,g_r]$. 

We denote by $H^2(\CC^m,\mu_G)$ the Hardy space of vector-valued functions with  the norm 
\[
\|q\|^2_{H^2(\CC^m,\mu_G)}:=\inpi \| G(e^{i\theta}) q\|^2_{\CC^m} \ d\theta.
\]
 Remember that $\GG=\span{S^k g_\ell\ :\ \ell\leq r,k\geq 0}$. Let $f=Gq$ be in $H^2(\CC^m)$. This forces $q$ to be in $H^2(\CC^m,\mu_G)$.  The measure $\mu_G$ will play a role in section~\ref{matricialhayashi}.
 
  We would like to build from $G$ the functions $F$ and $B$, the analogues of those appearing in Theorem~\ref{sarason}. After, we will show that $T_{Id_r-B}T_{G^*}$ is an coisometry from $\GG$ to $\HH(B)$, or equivalently, $q\mapsto (T_{Id_r-B}T_{G^*}G)q$ is an coisometry from $H^2(\CC^m,\mu_G)$ to $\HH(B)$. As a consequence, we will obtain the following equality, the key of the proof of the generalization of Theorem~\ref{sarason}~:
\[
Id_r -T_{B}T_{B^*}=\left(T_{Id_r-B}T_{G^*}\right)\left(T_{Id_r-B}T_{G^*}\right)^*.
\] 
 We begin by defining $F$ the analytic function on the disc by 
$$
\forall z\in \DD, \ \ F(z):=\inpi \frac{e^{i\theta}+z}{e^{i \theta}-z} G(e^{i\theta})^* G(e^{i\theta}) d\theta.
$$
Note that, if $G=U G'$, where $U$ is inner of rank $m$, then the functions   $F$  and $F'$ are the same.
Because the $(g_k)_{k\leq r}$ form an orthogonal basis of $W$, the coefficient $F(0)_{i,j}$ is $\< g_i , g_j \>_{H^2}$. So, $F(0)=Id_r$.
 For $z_0\in \DD$, let $u\in \CC^r$ be an eigenvector of the matrix $F(z_0)$. Then $\Re( \< F(z_0) u, u \> )$ is a Poisson integral, so  the real parts of the eigenvalues of $F(z_0)$ are $\|G(z_0)u\|^2\geq 0$. This implies that the moduli of the eigenvalues of $F(z_0)+Id_r$ are greater than 1, so $F(z_0)+Id_r$ is invertible.

Next, we define $B$, the matrix-valued Herglotz integral of $\mu_G$, by
$$
B(z):=(F(z)+Id_r)^{-1}(F(z)-Id_r).
$$
Because $F(0)=Id_r$, the function $B$ vanishes in zero. For all $u\in\CC^r$, \[ \|( F(z) \pm Id_r)u\|^2=\| F(z) u \|^2+\| u\|^2\pm 2 \Re \< F(z)u,u\>. \] Then, because $\Re \< F(z)u,u\> \geq0,$ 
\[ \|( F(z) +Id_r)u\|^2\geq \|( F(z) -Id_r)u\|^2\] and  $B$ lies in the unit ball of $H^\infty(\CC^r \to \CC^r)$. We can therefore consider $\HH(B)$.

\begin{lem}\label{lem5.2}
For all  $u,v\in \CC^m$ we have:
 \[ \<G k_w u,G k_z v\>_{H^2}=\<k_w^B (Id_r-B(w)^*)^{-1}u, k_z^B (Id_r-B(z)^*)^{-1}v\>_{\HH(B)}.\]
\end{lem}

\begin{proof}
 For all $u,v\in \CC^m$, we express the inner product	$\<G k_w u,G k_z v\>_{H^2}$ in terms of $F$:
	\[
	\begin{aligned}
\<G k_w u,G k_z v\>_{H^2}=&\frac{1}{2\pi} \int_0^{2\pi} \dfrac{1}{(1-\bar{w}e^{i\theta})(1-ze^{-i\theta})}\< G(e^{i\theta})u,G(e^{i\theta})v\>_{\CC^m} d\theta\\
	&=\dfrac{1}{2\pi(1-\bar{w}z)}\int \frac{1}{2} \left[\dfrac{e^{-i\theta}+\bar{w}}{e^{-i\theta}-\bar{w}} +\dfrac{e^{i\theta}+z}{e^{i\theta}-z}\right] \< G(e^{i\theta})u,G(e^{i\theta})v\>_{\CC^m} \ d\theta\\
	&=\dfrac{1}{2(1-\bar{w}z)}\<(F(w)^*+F(z))u,v\>_{\CC^m}.\\
	\end{aligned}
	\]
Because $(Id_r+B(z))$ and $(Id_r-B(z))^{-1}$ commute, 
		\[
	\begin{aligned}
F(w)^*+F(z)=&(Id_r-B(z))^{-1}(Id_r+B(z))+(Id_r+B(w)^*)(Id_r-B(w)^*)^{-1}\\
=&2(Id_r-B(z))^{-1} \left[Id_r-B(z)B(w)^*\right](Id_r-B(w)^*)^{-1}.\\
	\end{aligned}
	\]
 Finally, we interpret $\<G k_w u,G k_z v\>_{H^2}$ in term of  inner product of $k_w^B(z)$: 
	\[
	\begin{aligned}
\<G k_w u,G k_z v\>_{H^2}=&\dfrac{1}{2(1-\bar{w}z)}\<(F(w)^*+F(z))u,v\>_{\CC^m}\\
	&=\dfrac{1}{1-\bar{w}z}\<(Id_r-B(z))^{-1} \left[Id_r-B(z)B(w)^*\right](Id_r-B(w)^*)^{-1}u,v\>_{\CC^m}\\
	&=\<(Id_r-B(z))^{-1} k_w^B(z) (Id_r-B(w)^*)^{-1}u,v\>_{\CC^m}\\
	&=\< k_w^B(z) (Id_r-B(w)^*)^{-1}u,(Id_r-B(z)^*)^{-1}v\>_{\CC^m}\\
	&=\< k_w^B (Id_r-B(w)^*)^{-1}u,k_z^B(Id_r-B(z)^*)^{-1}v\>_{\HH(B)}.\\
	\end{aligned}
	\]

\end{proof}

The following lemma is useful in connection with de Branges--Rovnyak spaces (\cite{sarason:book} I-5):
\begin{lem}[Douglas's criterion]\label{douglas}
Let $H$, $H_1$ and $H_2$ be Hilbert spaces, and let $A:H_1\to H$,  $B:H_2\to H$ be contractions. We define $\MM(A):=AH_1$ and $\MM(B):=BH_2$. Then $\MM(A)=\MM(B)$ is equivalent to $AA^*=BB^*$.
\end{lem}

Remember that $\GG=\span{S^k g_\ell: \ell\leq r, k\in \NN }=\span{ Gk_w u : u\in \CC^r, w\in \DD }$ and that $\Theta$ is an inner function  such that $\GG=\Theta H^2(\CC^m)$. (When $G$ is outer, $\Theta$ is a constant unitary-matrix)

\begin{lem}\label{lem1}\mbox{}
\begin{enumerate}
\item For all $u\in \CC^m$,  $ T_{Id_r-B}T_{G^*}$ maps $Gk_wu$ to $k^B_w (Id_r-B(w)^*)^{-1}u$.
\item If $G$ is outer, then $T_{Id_r-B}T_{G^*}$ is an isometry from $\GG$ onto $\HH(B)$.
\item If $G=G_i G_o$, with $G_i$ inner and $G_o$ outer, then $T_{Id_r-B}T_{G^*}$ is a coisometry of $\GG$ to $\HH(B)$ with null space $K_{G_i}\cap \GG$.
\item Define $\MM(T_{Id_r-B} T_{G^*}):=T_{Id_r-B} T_{G^*}\GG$. Equipped with the inner product
\[
\<T_{Id_r-B} T_{G^*} h_1,T_{Id_r-B} T_{G^*}h_2 \>:=\<h_1,h_2 \>_2  \ \ \ \ \forall h_1,h_2\in \GG\cup (\ker T_{Id_r-B} T_{G^*})^\perp,
\]
$\MM(T_{Id_r-B} T_{G^*})$ coincides with the de Branges--Rovnyak space  $\HH(B)$.

\end{enumerate}
\end{lem} 

\begin{proof} 
\begin{enumerate}
	\item  We begin by computing the range of $Gk_w u$ by $T_{Id_r-B}T_{G^*}$:
	\[
	\begin{aligned}
	\<\left(T_{Id_r-B}T_{G^*}\right)Gk_w u,k_z v\>_{H^2}&=\< (Id_r-B(z)) T_{G^*}G(z) k_w(z) u, v\>_{\CC^m}\\
	&=\<  T_{G^*}G  k_w  u,k_z (Id_r-B^*) v\>_{H^2} \\
	&=\< G k_w  u, G k_z (Id_r-B(z)^*) v\>_{H^2}\\
	&=\< k_w^B (Id_r-B(w)^*)^{-1} u,k_z^B v\>_{\HH(B)}\\
	&=\< k_w^B (Id_r-B(w)^*)^{-1} u,k_z v\>_{H^2}.\\
	\end{aligned}
	\]
	Therefore, $\left(T_{Id_r-B}T_{G^*}\right)$ sends $Gk_w u$ to $k_w^B (Id_r-B(w)^*)^{-1}u.$ 

\item  The inner product of two functions in $\GG$ is equal to the inner product of their images in $\HH(B)$:
\[
	\begin{aligned}
\<T_{Id_r-B}T_{G^*}Gk_w u,T_{Id_r-B}T_{G^*}Gk_z v\>_{H^2}=&\< k_w^B (Id_r-B(w)^*)^{-1} u,  k_z^B (Id_r-B(z)^*)^{-1} v\>_{\HH(B)}\\
	=&\<Gk_wu,Gk_zv \>_{H^2}.
	\end{aligned}
\]
 Because $G$ is outer, the functions $Gk_w u$ span $\GG$, reduced to $H^2(\GGG)$ and the result follows.

\item  The definition of $B$ does not depend on $G_i$, so $T_{Id_r-B}T_{G^*}=T_{Id_r-B}T_{G_o^*}T_{G_i^*}$. But $T_{Id_r-B}T_{G_o^*}$ sends $T_{G_i^*}\GG$ isometrically to $\HH(B)$, which is dense in $\GG_o$, so we get the result by continuation. Moreover,
$\ker T_{Id_r-B}T_{G^*}|_{\GG}=\ker T_{G_i^*}\cap \GG=K_{G_i}\cap \GG$.

\item  The last sentence allows us to identify the two de Branges--Rovnyak spaces:
\[
\MM\left(T_{Id_r-B}T_{G^*}\right):=T_{Id_r-B}T_{G^*} \GG=\HH(B)=\MM\left((Id_r-T_BT_{B^*})^{1/2}\right).
\]

\end{enumerate}
\end{proof}

\begin{thm}\label{key}As operators on $H^2(\CC^r)$, we have $(T_{Id_r-B}T_{G^*})(T_{Id_r-B}T_{G^*})^*=Id_r-T_B T_{B^*}$.
\end{thm}

\begin{proof}
The Douglas criterion, Lemma \ref{douglas}, implies that $\MM\left(T_{Id_r-B}T_{G^*}\right)=\HH(B)$ is equivalent to $(T_{Id_r-B}T_{G^*})(T_{Id_r-B}T_{G^*})^*=Id_r-T_B T_{B^*}$ as operators on $H^2(\CC^r)$.

 \end{proof}

\subsection{A matricial  version of Sarason's theorem}
 
With the previous notation,

\begin{thm}\label{theoprincipal}
Let $G=[g_1,\dots,g_r]\in H^2(\CC^r\to \CC^m)$  and let $U\in H^\infty(\CC^r\to \CC^r)$ be inner of rank $r$ vanishing at zero. Then
\[
T_G|_{K_U} \mbox{ is an isometry }\Longleftrightarrow T_{B^*}K_U=\{0\} \Longleftrightarrow BH^2(\CC^r)\subset UH^2(\CC^{r}).
\]
\end{thm}

The proof follows Sarason's ideas, with modifications to bypass the fact that $T_{B^*}K_U$ might not be a subspace of $K_U$.

\begin{proof}
The last equivalence is obvious.

Suppose that $T_{B^*}K_U=\{0\}$. Then $T_Gh=T_G T_{Id_r-B^*}h$ for all $h\in K_U$.
Thanks to Lemma \ref{lem1}, $(T_{Id_r-B} T_{G^*})(T_{Id_r-B} T_{G^*})^*=Id_r-T_BT_{B^*}$, and we compute the norm of $\| T_G h \|^2_{H^2}$:
	\[
	 \begin{aligned}
	\| T_G h \|^2_{H^2}&=\| T_G T_{Id_r-B^*} h\|^2_{H^2}=\< (T_{Id_r-B} T_{G^*})(T_{Id_r-B} T_{G^*})^*h,h\>\\
	&=\< (Id_r-T_BT_{B^*})h,h\>\\
	&=\<h,h\>_{H^2}=\|h\|^2_{H^2}.
	 \end{aligned}
	\]
	Thus, $T_G$ acts as an isometry on $K_U$. 
 
 Conversely, suppose that $T_G|_{K_U}$ is an isometry.  Let $h\in K_U$. Lemma \ref{lem1} and  Theorem~\ref{key} assert that $T_{Id_r-B}T_{G^*} : \GG \to \HH(B)$ is a coisometry, and\\ \mbox{ $(T_{Id_r-B}T_{G^*}) (T_{Id_r-B}T_{G^*})^*=Id_r-T_B T_{B^*}$} on $H^2(\CC^r)$. It follows that\\ \mbox{$\|T_G T_{Id_r-B^*} h \|_{H^2}^2=\<(Id_r-T_B T_{B^*})h,h \>_{H^2}$}, whose development is:
 \[ 
 \|T_Gh\|^2-\<T_{G}T_{B^*}h,T_Gh\>-\<T_Gh,T_{G}T_{B^*}h\>+\|T_{G}T_{B^*}h\|^2=\|h\|^2-\|T_{B^*}h\|^2.
\]

Using the hypothesis $\|T_Gh\|=\|h\|$, we get
\begin{equation}\label{calcul1}
 \<T_{G}T_{B^*}h,T_Gh\>+\<T_Gh,T_{G}T_{B^*}h\>=\|T_{B^*}h\|^2+\|T_{G}T_{B^*}h\|^2.
\end{equation}

Now, with Sarason's trick, we will show that $T_G T_{B^*}h=0$, for $h\in K_U$. Remember that $U(0)=0$, so $\CC^r\subset K_U$ and because $B(0)=0$, we get $T_{B^*}v=0$ for all $v\in\CC^m$.
With $c\in \CC$ and $v\in\CC^m$, $h+cv$ stays in $K_U$ and $T_{B^*}(h+cv)=T_{B^*}h$.  Replacing $h$ by $h+cv$ in the equality~\ref{calcul1}, we have:
\[
\<T_{G}T_{B^*}h,T_G(h+cv)\>+\<T_G(h+cv),T_{G}T_{B^*}h\>=\|T_{B^*}h\|^2+\|T_{G}T_{B^*}h\|^2.
\]
 This is equivalent to \[ 2\Re\left( c \<T_{G}T_{B^*}h,T_Gv\>\right) =\|T_{B^*}h\|^2+\|T_{G}T_{B^*}h\|^2-2\Re \<T_GT_{B^*}h,T_G h  \>.\]
This holds for all $c\in \CC$, so necessarily
\begin{equation}\label{eq}
  \Re{\<T_{G^*}T_{G}T_{B^*}h(0),v\>}=0 \mbox{ and } \|T_{B^*}h\|^2+\|T_{G}T_{B^*}h\|^2-2\Re \<T_{G^*GB^*}h, h  \>=0.
\end{equation}
The first equality holds for all $v\in\CC^m$, so $T_{G^*}T_{G}T_{B^*}h(0)=0$. Replacing $h$ by $S^{*k}h$, which stays in $K_U$, we deduce that  
$T_{G^*}T_{G}T_{B^*}S^{*k}h(0)=0$ and so $T_{G^*}T_{G}T_{B^*}h=0$. 
This implies that $T_{B^*}h\in \ker T_G$ or $T_{G}T_{B^*}h\in \ker T_{G^*}$. We denote $f=T_{G}T_{B^*}h$. Then, $f\in \ker T_{G^*}\cap \GG$ and there exists $q\in H^2(\CC^m,\mu_G)$ such that $f=Gq$. The norm of $f$ is $\|f\|^2=\<G^*Gq,q\>=\<T_{G^*}f,q\>=0$. Finally, $T_{B^*} K_U\subset \ker T_G$.

The second equality of (\ref{eq}) implies the following equivalences:
\[
\begin{aligned}
&\|T_{B^*}h\|^2+\|T_{G}T_{B^*}h\|^2-2\Re \<T_{G^*}T_GT_{B^*}h, h  \>=0\\
\Longleftrightarrow & \|T_{B^*}h\|^2-\|T_G h\|^2+ \|T_G h\|^2 +\|T_{G}T_{B^*}h\|^2-2\Re \<T_{G^*}T_GT_{B^*}h, h  \>=0\\
\Longleftrightarrow & \|T_{B^*}h\|^2-\|T_G h\|^2+ \|T_G T_{Id_r-B^*}h\|^2=0\\
\Longleftrightarrow & \|T_G T_{Id_r-B^*}h\|^2=\|T_G h\|^2-\|T_{B^*}h\|^2.\\
\end{aligned}
\]
 But we know that $T_{B^*} K_U\subset \ker T_G$, so $\|T_G T_{Id_r-B^*}h\|^2=\|T_G h\|^2$ and $\|T_{B^*}h\|^2=0$.  So we get $T_{B^*}K_U=\{0\}$ as desired.
 
\end{proof}

\begin{cor}
The operator $T_{Id_r-B}T_{G^*}$ acts on $\FF$ as  division by $G$.
\end{cor}

\begin{proof}
Let $Gh\in\FF$. Then, thanks to the last theorem, $T_{B^*} h=0$. So, \[ T_{Id-B}T_{G^*} Gh=T_{Id-b}T_{G^*}T_G T_{Id-B^*}h\] and Lemma \ref{key} implies that $(Id-T_{B}T_{B^*})h=h$.
\end{proof}

In the original proof, Sarason uses the fact that scalar  model spaces $K_u$ (or more generally  de Branges spaces \cite{sarason:book} II-7) are stable under the action of $T_{\bar{b}}$ for every symbol  $b\in H^\infty$. This does not hold for matrix symbols. The inclusion  $T_{B^*}K_U\subset K_U$ means that $BU H^2 (\CC^r)\subset UH^2(\CC^r)$, and so $U^*B U \in H^{\infty}(\CC^r\to\CC^r)$.  This is obvious if $B$ and $U$ commute.

In this section, we will construct an example in $H^\infty(\CC^2\to\CC^2)$ where  $T_{B^*}K_U$ is not contained in $K_U$.
The following characterization of $(2\times 2)$-matrix-valued inner functions is due to Garcia, in \cite{garcia}.
\begin{thm}\label{garcia}
Let  $U\in H^\infty(\CC^2\to \CC^2)$. Then $U$ is inner if and only if $U$ is of the form:
	\[ 
U=\begin{pmatrix}
a  & -b \\ \theta \bar{b} & \theta \bar{a}
\end{pmatrix}\]
where $\theta:=\det \ U$  is inner, and $ a,b\in K_{z\theta}$ verify $ |a|^2+|b|^2=1$  a.e.\ on $\TT$.
\end{thm}

   Garcia gives an interesting example of an inner function by taking $a:=(1+\theta)/2$,  $b:=-i(1-\theta)/2$ and
\[
V:=	\frac{1}{2}
	\begin{pmatrix}
(1+\theta) & i(1-\theta) \\ -i(1-\theta)  &(1+\theta) 
\end{pmatrix}. 
	\]

Taking $\theta=z$, for example, we notice that the entries are outer scalar  functions. We can look for $U=zV$ which is still inner and vanishing at zero.

Let $B=\begin{pmatrix} b_1 & b_2 \\ b_3 & b_4 \end{pmatrix}\in H^\infty (\CC^2\to \CC^2)$.
A calculation shows that $U^*BU$ is equal to:
{\small
\[
\begin{pmatrix}
b_1+b_4+\Re(\theta)(b_1-b_4)+i \Im \theta (b_3+b_2) & b_2-b_3+\Re(\theta)(b_3+b_2)+i \Im \theta (b_1+b_4)\\
b_3-b_2+\Re(\theta)(b_3-b_2)+i \Im \theta (-b_1+b_4) & -b_1+b_4+\Re(\theta)(b_1+b_4)+i \Im \theta (-b_3-b_2)
\end{pmatrix}.
\]}
If we suppose that $b_4=-b_1$ and $b_3=-b_2$, then 
\[
U^*BU=\begin{pmatrix}
\Re(\theta)b_1 & b_2+\Im(\theta)b_1 \\
-2b_2+ \Im (\theta) b_1 & -\Re(\theta)b_1 
\end{pmatrix},
\] 
which is not in  $H^\infty(\CC^2\to \CC^2)$ and so $T_{B^*} K_U \not\subset K_U$. 

\begin{rem}
In \cite{sarason:art}, Sarason establishes an alternative proof of Theorem~\ref{descriptionnearly} using Corollary \ref{key}. This approach could be generalized to the vector-valued case.

\end{rem}

\section{Kernel of Toeplitz operators}\label{kernel}

\subsection{Some examples}
For a nearly $S^*$-invariant subspace $\FF=GK_U$, we recall that $W=\FF\cap(\FF\cap zH^2(\CC^m))^\perp$, and  $r:=\dim\, W\leq m$. 
If $m=2$, then we have two ways to build $\FF$ with $\dim\, \FF=2$.

\begin{ex}~\label{ex1}
	If $r=2$, $G=[g_1,g_2]$ and $U=z Id_2$. Let  $\FF$ be the space
	\[ \left\{\left(\begin{array}{c} a(1+z)^{1/2} \\ b(1-z)^{1/2} \end{array}\right), (a,b)\in \CC^2\right\}.\]
 Because $f(0)=0$ implies $f=0$, it is  $S^*$-nearly invariant. We see that $W=\FF$ and $\FF=T_G \CC^2$ with 
	\[
	G(z)=\dfrac{1}{2} \begin{pmatrix}  (1+z)^{1/2} & 0 \\ 0 & (1-z)^{1/2}   \end{pmatrix}.
	\]
The functions $\frac{1}{2}(1+z)^{1/2}$ and $\frac{1}{2}(1-z)^{1/2}$ are outer in $H^2$ and $G$ is outer, because its determinant is outer (see section \ref{innerouter}).  Moreover, $\GG=H^2(\CC^2)$. 
	
	Is $\FF$ the kernel of a Toeplitz operator $T_{\phii}$? We will build  $\phii\in L^\infty(\CC^2\to \CC^2)$ as the following. Remark that $G(e^{it})^*G^{-1}(e^{it})$ is  diagonal. The diagonal terms are $e^{-\frac{1}{2}it}$ and $-e^{\frac{1}{2}it}$. So, $G^*G^{-1}$ lies in $L^\infty (\CC^2\to\CC^2)$ and then
	\begin{equation}\label{arg}
	T_{\bar{z} G(z)^*G(z)^{-1}}=p_+  \begin{pmatrix}  z^{-\frac{3}{2} }  & 0 \\ 0 & -z^{-\frac{1}{2}}   \end{pmatrix} \ \ a.e.\ z\in \TT.
	\end{equation}
	Every $f$ in $\FF$ is of the form $f=G e$, with $e\in\CC^2$, and   $T_{\bar{z} G^* G^{-1}} f=p_+(\bar{z} G^* e)=0$. So $\FF$ is the kernel of $T_{\bar{z} G^* G^{-1}}$.
\end{ex}
	
\begin{ex}
	We modify the previous example to get a nearly $S^*$-invariant subspace which is not the kernel of a Toeplitz operator. Let $\FF$ and $G$ be defined by 
	 \[ \FF=\left\{\left(\begin{array}{c} a(1+z)  \\ b(1-z)  \end{array}\right) (a,b)\in \CC^2 \right\} \mbox{ and } G(z)=\frac{1}{\sqrt{2}}\begin{pmatrix}   1+z   & 0 \\ 0 &  1-z    \end{pmatrix}.\]
	  We will show that if it is the kernel of a Toeplitz operator,  it is also the kernel of the Toeplitz operator with symbol $\phii(e^{it}):= G^*(e^{it})U^*(e^{it})G^{-1}(e^{it})$. This symbol is a diagonal matrix. The diagonal terms are $e^{-2it}$ and  $-1$ a.e.\ $e^{it}\in\TT$. But 
	 \[ \ker T_\phii=\left\{\left(\begin{array}{c} a+bz  \\ 0 \end{array}\right)\ : \ (a,b)\in \CC^2 \right\}\not=\FF,\]
	  and so $\FF$ fails to be the kernel of a Toeplitz operator. 
 \end{ex}
	
\begin{ex}
	 Let $r=1$, $G=[g_1]$ and $\dim\, K_U=2$. Let $\FF$ be defined by
	 \[ \FF:= \left\{\left(\begin{array}{c} a+bz \\ 0 \end{array}\right)\, :\, (a,b)\in\CC^2\right\}.\] 
	 This space is the nearly $S^*$-invariant $\FF=GK_U$ with   $G(z)=\left(\begin{array}{c} 1 \\ 0 \end{array}\right)$ and $U(z)=z^2.$ 
	With the  notation defined in section \ref{innerouter}, we have
	\[
	\GG:=\span{S^k G}=\begin{pmatrix}1&0\\ 0&0\end{pmatrix} H^2(\CC^2), \ \Theta_0=\begin{pmatrix}1 \\  0\end{pmatrix}, \ \mbox{ and }  \GGG=\begin{pmatrix}\CC \\  0\end{pmatrix}.
	\]
	Because $\GG=H^2(\GGG),$ the function $G$ is outer.  As we saw before, $\tilde{G}$  is the $(1\times 1)$-square matrix  such that $G=\Theta_0 \tilde{G}$. Here,  $\tilde{G}=1$. Furthermore, $\FF=\ker T_\phii$ where  
	\[
	\phii:=\begin{pmatrix} \tilde{G}^* U^*  \tilde{G}^{-1} & 0 \\ 0 & 1 \end{pmatrix}=\begin{pmatrix}  \bar{z}^2 & 0 \\ 0 & 1 \end{pmatrix}.
	\]
 This example illustrates the problem  of dimensions : the interesting part lies in $\GGG$ which is a subspace of $\CC^m$.
\end{ex}

\subsection{The case  $r=m$}
This section treats the particular case where $r=m$ seen in the Example~\ref{ex1}.

First of all, a nearly $S^*$-invariant subspace $\FF$ which is the kernel of a Toeplitz operator $T_\phii$ has the form $\FF=T_G K_U$ with $G$ outer. Remember that the columns of $G\in H^2(\CC^m\to\CC^m)$ form an orthogonal basis of $W:= \FF\cap \left\{ \FF  \cap zH^2(\CC^m) \right\}^\perp$. Because $\det G\not\equiv 0$, the inner-outer factorization for matrix-valued functions allows us to factorize with an inner function on the right. Let $G_o$ be outer, and $G_i$ be inner, such that $G=G_o G_i$.  Because $U(0)=0$, it follows $K_U$ contains $\CC^m$ and $G\CC^m\subset \ker T_\phii$. For all $e\in \CC^m$, we have $Ge\in \ker T_\phii$. So, there exists $H\in H^2(\CC^m\to \CC^m)$ such that $\phii G=\bar{z} H^*$. But $G=G_o G_i$, then $\phii G_o G_i = \bar{z} H^*$ and $\phii G_o = \bar{z} H^* G_i^*$. Finally, $T_{\phii}( G_o e)=0$, which implies that the columns of $G_o$   form an orthonormal basis of $W$. A nearly $S^*$-invariant subspace  $\FF$ which is the kernel of a Toeplitz operator is of the form $G K_U$, with $G$ outer.

Following Sarason, the first step is to understand what happens if $\dim \ker \FF=m=r$.
A new notion, namely the rigid functions, appears to characterize when $\FF$ is the kernel of a Toeplitz operator. A scalar function $f\in H^1$ is said to be \textit{rigid} if the only functions in $H^1$ which have the same argument are of the form $cf$ with $c$ a non-negative constant. 
 When $\det(F)\not=0$, we write $(R_F,A_F)$ for the polar decomposition of the matrix $F=R_F A_F$. The matrix $R_F$ is positive and $A_F$ is unitary. The matrix $A_F$  is called the \textit{argument} of $F$.

\begin{defn}
Let $F\in H^1(\CC^m\to \CC^m)$  be a square matrix-valued function. Then $F$  is \textit{rigid} if the only functions which have the same argument $A_F$ are of the form $RF$, where $R$ is a constant hermitian positive matrix.
\end{defn}

In fact, we can show that rigid functions are exactly the exposed points of the unit ball of $H^1(\CC^m\to \CC^m)$. We recall the definition. Let $X$ be a Banach space and $A$ a closed subset of $X$. A point $x\in A$ is an \textit{exposed point } of $A$ if there exists $L\in X^*$ such that $L(x)=1$ and $\Re(L(y))<1$ for all $y\in A\setminus\{x\}$. The functional $L$ associated to $x$ is unique. In our case, $L$ is:
\[
\begin{array}{rcl}
L:H^1(\CC^m\to\CC^m)&\longrightarrow & \CC \\
H&\longmapsto & \text{tr} \left( \dfrac{1}{2\pi} \int_0^{2\pi} H(e^{it})A_{F}(e^{it}) \ dt\right),
\end{array}
\]
where $\text{tr}$ denote the trace.  It is easy to verify that $L(F)=1$ and that $F$ is rigid if and only if $L$ is unique.  

Moreover, exposed points are extreme points (in the sense of  convexity). The extreme points of the unit ball of $H^1(\CC^m\to\CC^m)$ are the outer functions with norm 1. For more results see \cite{cambern} and \cite{beneker}.
	Before stating the next lemma, we define precisely the norm $\tn \cdot \tn$ which we shall use. Let $(e_k)_k$ be the canonical basis of $\CC^m$. Then $\tn \cdot \tn$  is the matricial norm defined by $\tn  A\tn^2:= \sum_{k=1}^m \< Ae_k,A e_k\>_{\CC^m}$.

	\begin{lem}\label{casm=r}
	Let $G\in H^2(\CC^m\to\CC^m)$ and $\FF=T_G \CC^m$ be a nearly $S^*$-invariant subspace of $H^2(\CC^m\to\CC^m)$. If $\FF$ is the kernel of a Toeplitz operator, then $G^2$ is rigid and $\FF=\ker\ T_{\bar{z} G^*G^{-1}}.$ 
	\end{lem}
	
	\begin{proof}
		Let $\phii$ be the  symbol of the Toeplitz operator for which $\FF$ is the kernel. Because $G\in \ker T_{\phii}$, there exists  $H\in  H^2(\CC^m\to\CC^m)$ such that $ \phii G=\bar{z} H^*$.
	We begin by showing that $H$ is outer. Let $H=VH_o$ the inner-outer decomposition of $H$. Then  $\phii G=\bar{z} H_o^* V^*$, so $\phii GV =\bar{z} H_o^*$. Finally, $GV\in \ker T_\phii$, but $\ker T\phii=G\CC^m$, so $V$ is constant.
	
	We would like to write $T_\phii$ as the product of two Toeplitz operators, such that the first one is injective and the symbol of the second one depends only on $G$. Because $G\in H^2(\TT,\CC^m\to\CC^m)$ is outer, $\det\ G(z)$ is a outer scalar function in $\NNN^+$ and for all $z\in\DD$, the inverse $G(z)^{-1}$ has a sense. The function $G^{-1}$ lives in $\NNN^+$. Now, we consider the polar decomposition of $G(z)=A_G(z)R_G(z)$ with $R_G(z)=(G(z)^*G(z))^{1/2}$  positive-definite hermitian   and $A_G(z)$ unitary. 
	We have
	\[
	G(z)^{-1}=A^*(z) R(z)^{-1}\mbox{ and }  G(z)^*=A^*(z) R(z).
	\]
	 Because $G$ is outer, $\GG=H^2(\CC^m)$. We can say that $G H^\infty(\CC^m)$ is dense in $\GG$. Let $f=Gf'\in G H^\infty(\CC^m)$. We define $\psi\in L^\infty(\TT,\CC^m\to \CC^m)$ by 
	\[
	\psi =\bar{z} H^* (G^{-1})^*	G^* G^{-1},
	\]
	and then $T_\psi f$ is exactly $T_{(G^{-1}H)^*} T_{ G^*\bar{z} G^{-1}}f$. We extend by continuity from\\ $G H^\infty(\CC^m)$ to $H^2(\CC^m)$.
	
	Because $G$ and $H$ are outer,  $G^{-1}$  and $G^{-1}H$ are outer in  $\NNN^+(\CC^n\to\CC^m)$. But $\phii G=\bar{z} H^*,$ so we have $\phii(z)=\bar{z} H^*(z) G(z)^{-1}$. So, for almost every  $e^{it}\in \TT$, the norm 	$\tn G(e^{it})^{-1} H(e^{it}) \tn$  is given by:
	\[
	\begin{aligned}
	\tn G(e^{it})^{-1} H(e^{it}) \tn  &= \tn H(e^{it}) H(e^{it})^{*-1} H^*(e^{it}) G(e^{it})^{-1}\tn \\
	&= \tn A_H(e^{it})^2 \ e^{it} \phii(e^{it})\tn \\
	&=\tn \phii (e^{it})\tn .
		\end{aligned}
	\]
	 This proves that $G^{-1} H \in H^\infty(\CC^m\to\CC^m)$ is outer with the same norm as $\phii$.  So, the kernel of the Toeplitz operator with symbol $G^{-1} H$ is trivial. By construction, $\ker T_{\bar{z} G^*G^{-1}}=\FF$.
	 
	 We shall prove that $G^2$ is rigid. 
	 
	 Let  $J$ be a function with the same argument as $G^2$. We can suppose that it is outer, because if not then we can build an outer function with the same argument. Indeed, if  $J=J_i J_e$, then $-(Id+J_i)^2(Id-J_i^*)^2=2Id-J_i^2-J_i^{*2}$ is positive and $J_i^{*}(Id+J_i)^2J$ or $-J_i^{*}(Id-J_i)^2J$ is outer with the same argument as $G^2$.
	 	  
	 	  Let $J_1=J^{1/2}=A_F(z)R_J(z)^{1/2}.$ The hypothesis on $J$ implies that $R_J\not\equiv R_{G^2}$, so it is the same with $F_1$ and $G$. Because $J_1^* J_1^{-1}=A_G^2=G^*G^{-1}$ for a.e.\ $e^{it}\in\TT$,  then for all $u\in\CC^m$, $T_{\bar{z} G^* G^{-1}}F_1u=0$. But $\ker T_{\bar{z} G^* G^{-1}}=G\CC^m$, so $J_1u\in G\CC^m$. The contradiction follows. By hypothesis, $F_1$ is not a  multiple of $G$ by a constant matrix.
	 	  	
	\end{proof}
	
The following consequence is an interesting characterization of the rigid functions. It is the matricial analogue of  a result of Sarason (\cite{sarason:book} chapter X page 70) used to show that $1+z$ is rigid.
	
	\begin{prop}\label{injectif}
	If $F\in H^2(\CC^m\to \CC^m)$ is outer, then $F^2$ is rigid  if and only if the Toeplitz operator   $T_{F^*F^{-1}}$ has a trivial kernel.	
	\end{prop}
	
	\begin{proof}
	Because $\ker \ S^*=\CC^m$ and $T_{\bar{z} F^*F^{-1}}=S^* T_{F^*F^{-1}}$, then $\ker \ T_{F^*F^{-1}}$ is trivial          is equivalent to $\dim \ker \ T_{\bar{z} F^*F^{-1}}=m$. But, thanks to the previous lemma, this is true only if $F^2$ is rigid.
		
	\end{proof}
	
	Here is the matricial analogue of Lemma 1 of \cite{sarasontoeplitz} page 161 for an outer square matrix $G$.  
	\begin{lem}\label{casgeneralr=m}
	 Let $G\in H^2(\CC^m\to \CC^m)$ be outer and let $U\in H^\infty(\CC^m\to\CC^m)$ be inner vanishing at zero. Let $\FF=T_G K_{U}$ be the kernel of a Toeplitz operator. Then it is the kernel of $T_{G^*U^* G^{-1}}$.	
	\end{lem}

	\begin{proof}	This proof follows Sarason's. The fact that $\FF$ is the kernel of a Toeplitz  $T_\phii$ allows us to show that $G^*U^*G^{-1}$ defines a symbol in $L^\infty(\TT,\CC^m\to\CC^m)$. We begin the proof by building an outer function with the same norm as $\phii$, then we can suppose that $\phii$ takes values in the set of norm-$1$ matrices. Let $(e_k)_{1\leq k \leq m}$ the canonical basis of $\CC^m$.
	
 Because $p_+(\phii g_k)=0$, there exists $h_k\in H^2(\CC^m)$ such that $\phii G e_k=\bar{z} \overline{h_k}$ for $k\in\{1,\dots,m\}$.
The norm	$\tn \phii(e^{it})G(e^{it})\tn^2$ is equal to $\sum_{k=1}^m\<h_k,h_k\>$, so  $\log \tn \phii(e^{it})G(e^{it})\tn =\dfrac{1}{2} \log \left(\sum_{k=1}^m\<h_k,h_k\>\right).$

	But the function  $e^{it}\mapsto \sum_{k=1}^m\<h_k(e^{it}),h_k(e^{it})\>$ is in $H^1$, so it is $\log$-integrable, just as  $e^{it}\mapsto\tn \phii(e^{it})G(e^{it})\tn $. 
	We deduce from \[\tn \phii(e^{it})G(e^{it})\tn \leq \tn \phii(e^{it})\tn \tn G(e^{it})\tn\] that  $e^{it}\mapsto \tn \phii(e^{it})\tn $ is $\log$-integrable.
	Via the Poisson kernel (cf \cite{katsnelson}), we build $\psi \in H^\infty(\CC^m\to\CC^m)$ outer  with the same norm:
	\[
	\psi(z):=\exp \left(\dfrac{1}{2\pi} \int_\TT \dfrac{e^{it}+z}{e^{it}-z}\ \log\, \left( \phii^*(e^{it}\phii^*(e^{it})\right)^{1/2}\ dt\right).
	\]
	Because $\psi$ is outer, so is its determinant and the matrix $\psi(z)$ is invertible  for $z\in \DD$ and the  inverse is in $\NNN^+$. Then the radial limits exist almost everywhere and $\psi^{-1*}\phii \in L^\infty(\TT,\CC^m\to \CC^m)$. The values $\psi^{-1*}(e^{it})\phii(e^{it})$ are matrices with norm 1 for a.e.\ $e^{it}\in \TT$.

	So, there exists $\chi \in L^\infty(\TT,\CC^m\to\CC^m)$ which takes  values in the set of norm-$1$ matrices  such that $\phii=\psi^* \chi$. But, $\ker T_{\psi^*}$ is trivial, so $\ker T_\phii=\ker\ T_\chi$ and even if we replace $\phii$ by $\chi$, we can suppose that $\phii$ take values that are matrix of norm 1.
	
	Because $p_+(\phii G)=0,$ there exists $H\in H^2(\CC^m\to \CC^m)$ such that $\phii G= H^*$ and $H(0)=0$. 
	
	We note $H=H_i H_o$, and we will show  that $H_i=VU$ with $V$ a unitary constant matrix.
	
	Now, we prove that $U$ divides $H_i$. For all $ h\in K_U,$ $Gh\in \ker  T_\phii$  implies that $\phii G h$, which is equal to $H_o^* H_i^* h$, lies in $\overline{z H^2(\CC^m)}$. Because $H_o$ is outer,  $H_i^* h \in \overline{\; z \; H^2(\CC^m)}$ and then  $h$ is orthogonal to $H_i H^2(\CC^m)$. So, we have $K_U\subset K_{H_i}$.
	
	The converse inclusion holds: Let $h\in K_{H_i}$ be bounded. Then, there exists $H=H_i H_o\in H^2(\CC^m\to \CC^m)$ such that $T_\phii Gh=p_+(  H_o^* H_i^*h)$. Then for all $k\in H^\infty(\CC^m)$,  we have
\[  \< T_\phii Gh,k\>=\<  H_o^* H_i^* h,k\>=\<h,H_i H_o   k\>=0.\] 
Because the bounded functions are dense in $K_{H_i}$, we obtain that  $K_{H_i}\subset K_U$ and so $H_i=VU$ with $V$ unitary constant.
	
	As in the proof of Lemma \ref{casm=r}, we write $\phii =  H_o^* U^*G^{-1}$. Then \[ p_+(\phii f)=p_+(H_o^* (G^{-1})^*   G^* U^*G^{-1} f),\]  and $T_\phii =T_{(G^{-1}H_o)^*} T_{  G^* U^*G^{-1}}$.

	To conclude, we need to show that $ G^{-1}H_o $ lies in   $ H^\infty(\CC^m\to \CC^m)$.
	 The two functions $H_o$ and $G$ are outer, so $H_o^{-1}G $ is in the Nevanlinna-Smirnov class. 
	 To end the proof, we observe that $\tn G^{-1}H_o\tn=\tn \phii \tn$, which was done at the end of the proof of Lemma \ref{casm=r}. So  $\ker T_{(G^{-1}H_o)^*}$ is trivial, and $\ker T_\phii=\ker T_{{  G^* U^*G^{-1}}}$.

	\end{proof}

\subsection{The case $r<m$}

Let $G\in H^2(\CC^r\to \CC^m)$ be an outer function. With the notation of section \ref{innerouter}, we consider the space $\GG=H^2(\GGG)$, the unitary mapping  $\Theta_0: \CC^r \to \GGG$ and $G=\Theta_0 \tilde{G}$ with $\tilde{G}\in H^2(\CC^r\to \CC^r)$ outer.

Let $\Theta_1:\CCC^{m-r}\to \GGG^\perp$ be an unitary mapping. Then we decompose $H^2(\CC^m)$ as follows:
\[
H^2(\CC^m)=\GG\oplus\GG^\perp=H^2(\GGG) \oplus H^2(\GGG^\perp)=\begin{pmatrix}
\Theta_0 &0 \\ 0 & \Theta_1 
\end{pmatrix} \begin{pmatrix} H^2(\CC^r)\\ H^2(\CC^{m-r}) \end{pmatrix} .
\]
We denote $\Theta $ the $(m\times m)$-unitary matrix   with diagonal $\Theta_0,\Theta_1$.

	\begin{lem}\label{casr<m}
	Let $\FF=T_G \CC^r$ the kernel of a Toeplitz operator. We suppose that $\dim \FF=r$. Then $\tilde{G}^2\in H^1(\CC^r\to\CC^r)$ is a  rigid function and $\FF$ is the kernel of the Toeplitz operator with symbol
	\begin{equation}\label{eqphi}
	 \phi:=\Theta \begin{pmatrix} \bar{z} \tilde{G}^*\tilde{G}^{-1} & 0 \\ 0 & Id_{m-r} \end{pmatrix} \Theta^*.
	\end{equation}
	\end{lem}
	
	\begin{proof}
	Let $\phi$ in $L^\infty(\TT,\CC^m\to \CC^m)$ be defined as in (\ref{eqphi}).
	
If $f\in \GG^\perp$, then $\Theta^* f\in H^2(\{0_{\CC^r}\}\oplus H^2(\CC^{m-r}))$ and $T_\phi f=f$ , so $\ker \ T_\phi|_{\GG^\perp}=\{ 0\}.$ 
	
	If $f\in \GG$, then $\Theta^*f\in 	 H^2( \CC^{ r}\oplus \{0_{\CC^{m-r}}\}).$
	
	 We denote by $p_r$ the orthogonal projection from $\CC^m$ to $\CC^r\oplus \{0_{\CC^{m-r}}\}$. Let $\phii= p_r \Theta^* \phi \Theta p_r $, then $T_\phii$ is a Toeplitz operator on $H^2(\CC^r)$.  Write $\tilde{\FF}= \tilde{G}\CC^r$. We  verify that $\tilde{\FF}=p_r(\Theta^*G\CC^r)=p_r \Theta\FF$ is   nearly $S^*$-invariant of dimension $r$ and that $\ker T_\phii=\tilde{\FF}$.  Then, we apply Lemma~\ref{casm=r} to  $\tilde{\FF}$ in $H^2(\CC^r)$, which implies that 
$\tilde{G}^2$ is rigid and $\tilde{\FF}=\ker T_{\bar{z} \tilde{G}^*\tilde{G}^{-1}}$.

To conclude,
\[
\phi=\Theta \begin{pmatrix} \bar{z} \tilde{G}^*\tilde{G}^{-1} & 0 \\ 0 & Id_{m-r} \end{pmatrix} \Theta^*
\]
satisfies $\ker \ T_\phi=\FF$ and $\tilde{G}^2$ is rigid.

	\end{proof}

\begin{lem}\label{casgeneralr<m} 
	Let $\FF=T_G K_U$ be the kernel of a  Toeplitz operator.  Then $\tilde{G}^2\in H^1(\CC^r\to\CC^r)$ is a  rigid function  and $\FF$ is the kernel of the Toeplitz operator with symbol
	\[
	 \phi:=\Theta \begin{pmatrix}   \tilde{G}^*U^* \tilde{G}^{-1} & 0 \\ 0 & Id_{m-r} \end{pmatrix} \Theta^*.
	\]
	\end{lem}
	
	The proof uses the Lemmas \ref{casgeneralr=m} and \ref{casr<m}. The naive idea is to apply Lemma \ref{casgeneralr=m} to $\FF$, considered as a subspace of $\GG=H^2(\GGG)$. Then  we have $\dim \,W=r=\dim\, \GGG$. Unfortunately, the range of $T_\phi|_{\GG}$ is \textit{a priori} not in $\GG$.  We need to find a new  symbol $\phi'$ such that $T_{\phi'}|_\GG\to \GG$ and $T_{\phi'}|_{\GG^\perp}\to \GG^\perp$  with $\FF=\ker T_{\phi '}|_\GG$.
	
	\begin{proof}

 We begin by building a nearly $S^*$-invariant subspace $\FF'$ which will be the kernel of a certain  Toeplitz operator with symbol $\phi'$, such that $W'$ is of dimension $m$. It will be more convenient to write  $\FF=T_{G_0} K_{U_0}=\ker T_{\phi_0}.$ 

 By hypothesis,  $G_0\in H^2(\CC^r)$ is outer and its columns form an orthonormal basis of $W$. Let $(e_{r+k})_{k=1\dots m-r}$ be an orthonormal basis of $\GGG^\perp$. Because $\GG_0=H^2(\GGG)$, $G_1:=\frac{1}{\sqrt{2}}[(1+z)^{1/2} e_\ell]_{r+1\leq\ell\leq m}$ is an outer matrix and $\tilde{G}_1=\Theta_1^* G_1$ is square. This choice of $\frac{1}{\sqrt{2}}(1+z)$, which is a rigid scalar function, implies that  $\tilde{G}_1^2$ is rigid in $H^1(\CC^{m-r}\to \CC^{m-r})$.  Because  $\tilde{G}$, the  diagonal matrix with blocks $\tilde{G}_0,\tilde{G}_1$ is outer, $G:=[G_0,G_1]\in H^2(\CC^m\to \CC^m)$ is outer.

	Let $U_1:=zId_{m-r}$. It is inner and  $U_1(0)=0$ and $U'$, the matrix with  diagonal blocks $U_0$ and  $U_1$, is of rank $m$.  Finally, we can consider $\FF'=T_{G'} K_{U'}=T_{G_0}K_{U_0}\oplus^\perp T_{G_1}K_{U_1}$. But $T_{G_0}K_{U_0} \subset \GG$ and $T_{G_1}K_{U_1}\subset \GG^\perp$, so  $W'=W_0\oplus W_1$. We remark that $W_1=G_1 \CC^{m-r}$ and $\dim\, W'=m$.

	It remains to show that $\FF'$ is the kernel of a  Toeplitz operator with symbol $\phii$.
	First of all, because $\tilde{G}_1^2$ is rigid, Lemma \ref{casr<m} applied to $T_{G_1}K_{U_1}$ asserts that this space is the kernel of the  Toeplitz with symbol $\phi_1:=\Theta \begin{pmatrix}Id_{m-r} & 0 \\ 0 & \bar{z} \tilde{G}_1^*  \tilde{G}^{-1} \end{pmatrix} \Theta^*$. By hypothesis, $\FF$ is the kernel of $T_{\phii_0}$.
	Let $\phii':=\phii_0p_{\GG}+\phi_1 p_{\GG\perp}$, where $p_{\GG}$ is the matrix $\Theta \begin{pmatrix} Id_r & 0 \\ 0 & 0 \end{pmatrix} \Theta^*$ of projection from $H^2(\CC^m)$ to $\GG=H^2(\GGG)$.  It is clear that $\phii'\in L^\infty(\TT,\CC^m\to \CC^m)$ and that $\FF'=\ker \, T_{\phii'}$.

	Now, we can apply Lemma \ref{casgeneralr=m} to $\FF'$, the kernel of $T_{\phii'}$. So, $\tilde{G}'^2$ must be rigid. But it is a  block diagonal matrix, so $\tilde{G}_0^2$ must be rigid. Moreover, $\FF'$ is the kernel of $T_\phi'$ with
$\phi':=G^* U^* G^{-1}=\Theta \tilde{G}'^* U'^* \tilde{G}'^{-1} \Theta^*$. Finally, the symbol $\phi'$ is given by
	\[
\phi'	=\Theta \begin{pmatrix}
	\tilde{G}_0^* U_0^* \tilde{G}_0^{-1} & 0 \\ 0 & \bar{z} \tilde{G}_1^*\tilde{G}_1^{-1} 
	\end{pmatrix} \Theta^*.
	\]
	The diagonal  structure of $\Theta^* \phi' \Theta$ implies that the range of $T_{\phi'}|_{\GG_0}$ is contained in  $\GG_0$.

	To conclude, we will show that $\FF=\ker \, T_{\phi_0}$ where the symbol $\phi_0$ is:
	\[
	\phi_0=\Theta \begin{pmatrix}   \tilde{G}_0^*U_0^* \tilde{G}_0^{-1} & 0 \\ 0 & Id_{m-r} \end{pmatrix} \Theta^*\in L^\infty(\TT,\CC^m\to \CC^m).
	\]
 Applying Lemma \ref{casgeneralr=m} in $\GG_0=H^2(\GGG_0)$ with $\dim \GGG_0=r=\dim W$, we have  $\FF$ is the kernel of  $T_{ p_{\GG_0}\phi' p_{\GG_0}}$, so it is the kernel of the Toeplitz operator with symbol $\tilde{G}_0^* U_0 \tilde{G}_0^{-1}$. It remains to complete the symbol  on $\GG_0^\perp$. Then $\FF$ is the kernel of the Toeplitz operator with symbol
  \begin{equation}\label{symbol}
 \phi:=\Theta \begin{pmatrix}   \tilde{G}_0^*U_0^* \tilde{G}_0^{-1} & 0 \\ 0 & Id_{m-r} \end{pmatrix} \Theta^*.
 \end{equation}
\end{proof}

Let  $\FF$ be a nearly $S^*$-invariant subspace of the form $T_G K_U$ where $G$ is outer. 
We can summarize the section by saying that $\FF$ is the kernel of a Toeplitz operator if and only if $\tilde{G}^2$ is rigid. Furthermore, the formula \ref{symbol} express explicitely a symbol depending only on $G$ and $U$.

\section{A matricial version of Hayashi's theorem}\label{matricialhayashi}

The purpose of this section is to obtain a matricial version of Hayashi's theorem (\cite{hayashi}). This section treats mainly  the  case $r=m$.  We need some results about de Branges--Rovnyak spaces to state the theorem.

 The first chapter of Sarason's book \cite{sarason:book} is general enough for the matricial case. Sarason specializes  to the scalar case in the next chapters. The ideas come from the third and  fourth chapters. In this section, we adapt some of them to the matricial case. The main problem is the lack of commutativity. 

In section~\ref{matricialsarason}, we built from $G$ an outer function $B$ such that $\HH(B)$ is unitarily equivalent to $H^2(\CC^m,\mu_G)$. We saw that $B$ is in the unit ball of $H^\infty(\CC^m\to\CC^m)$ and a little investigation shows that $\log(I-B^*B)\in L^1(\TT,\CC^m\to \CC^m)$.  (Let $A$ be the outer function $2G(F+Id)^{-1}$, then for all $z\in\DD,$  $A^*(z)A(z)+B^*(z)B(z)=Id$, and the conclusion follows.)

 In this section, we need to reverse the construction, beginning with a $B$ in the unit ball of $H^\infty(\CC^m\to\CC^m)$ verifying $\log(I-B(e^{it})^*B(e^{it}))\in L^1(\TT,\CC^m\to \CC^m)$. (Such functions are non-extreme points of the unit ball of $H^\infty(\CC^m\to\CC^m)$, see Treil's theorem in \cite{Nik1}  page 85). 
 
 We define $A\in  \NNN^+(\CC^m\to\CC^m)$ to be the Poisson integral 
  \[ A(z):=\exp \left( \frac{1}{2\pi}\int_0^{2\pi} \dfrac{e^{it}+z}{e^{it}-z}\log(I-B(e^{it})^*B(e^{it}))\ dt \right).\]
By definition, $A$ is outer and $ A(e^{it})^*A(e^{it})+B(e^{it})^*B(e^{it})=Id\ \ a.e.\ on\ \TT$. Because $B$ is bounded, so is $A$.

The same calculation as the proof of Lemma~\ref{lem5.2} gives:
\[
\begin{aligned}
\Re \left((Id+B)(Id-B)^{-1}\right)&=\dfrac{1}{2}\left[(Id+B^*)(Id-B^*)^{-1}+ (Id+B)(Id-B)^{-1}\right]\\
&=(Id-B^*)^{-1}\left[Id-B^*B \right](Id-B)^{-1}\\
&=(Id-B^*)^{-1}A^*A(Id-B)^{-1}=R_{(Id-B)^{-1}A}^2.
\end{aligned}
\]
So, for $ u\in \CC^m,$  we have \[ \< \Re\left((Id+B)(Id-B)^{-1}\right)u,u\>_{\CC^m}=\|A(Id-B)^{-1} u\|_{\CC^m}^2 \]
 which is positive, so we can define the Herglotz integral corresponding to $B$ analogous to $F$ defined section~\ref{matricialsarason}.

 Let $\mu\in \MM(\TT)$ be the positive-definite-valued measure defined by
\[
\forall z\in \DD, \ \ F(z):=(Id+B(z))(Id-B(z))^{-1}=:iV+\int_{\TT} \dfrac{e^{i\theta}+z}{e^{i\theta}-z} d\mu(e^{i\theta}).
\]
Then, the Radon-Nikodym derivative of the absolutely continuous component of $\mu$ is $R_{(Id-B)^{-1}A}^2\in L^1(\TT,\CC^m\to \CC^m)$.

We define $G:=(Id-B)^{-1}A$. This function is outer, because its determinant is outer.  It is the quotient of two outer  functions in $H^\infty$, so $G\in \NNN^+$ (compute the inverse of $(Id-B)$ with the formula of the comatrix, whose coefficients are polynomials in the coefficients of $Id-B$, so are outer functions). The fact that $R_{(Id-B)^{-1}A}^2\in L^1(\TT,\CC^m\to \CC^m)$ implies that $G\in H^2(\CC^m\to \CC^m)$. If $\mu$ is absolutely continuous, then
\[
F(z):=(Id+B(z))(Id-B(z))^{-1}=iV+\int_{\TT}  \dfrac{e^{i\theta}+z}{e^{i\theta}-z} G^*(e^{i\theta}) G(e^{i\theta}) \dfrac{d\theta }{2\pi}.
\]
Finally, the hypothesis $B(0)=0$ implies $V=0$.  Let  $(e_k)_{1\leq k \leq m}$ be the canonical basis of $\CC^m$. 
The coefficients of $F(0)=Id=\int_\TT G^*(e^{i\theta}) G(e^{i\theta}) \dfrac{d\theta }{2\pi}$ are  the inner products $\< Ge_n,Ge_m\>_{H^2(\CC^m)}$, so  $G$ is a matrix which columns form an orthogonal basis and $G$ is of unit norm in $H^2(\CC^m)$.

In this context, we say that $(B,A)$ is a \textit{ corona pair}, or a \textit{pair}. When $\mu$ is absolutely continuous,   $(B,A)$ is said to be a \textit{ special pair}. Then, the measure $\mu_G$  defined in section \ref{matricialsarason} is equal to $\mu$ and $Gf\in H^2(\CC^m)$ is equivalent to $f\in H^2(\CC^m,\mu)$. 

 Recall that $G:=(Id-B)^{-1}A$ is outer, so thanks to Lemma \ref{lem1}-2,  the operator $T_{Id-B} T_{G^*}:H^2(\CC^m) \to \HH(B)$ is an isometry. To be a special pair means that for all $f\in H^2(\CC^m,\mu),$ $T_{Id-B} T_{G^*} Gf=f,$ the operator   $T_{Id-B} T_{G^*}$ represents the division by $G$. 
So, we can reformulate   Theorem~\ref{key}:

\begin{cor}
Let $(B,A)$ be a pair. Then $T_{Id-B}T_{G^*}$ is an isomerty from $H^2(\CC^m)$ to $\HH(B)$. The pair is special if and only if $T_{Id-B}T_{G^*}$ is surjective.
\end{cor}

The next Proposition is the analogue of the Proposition 6 in \cite{sarasontoeplitz}.
\begin{prop}[]\label{prop6}\mbox{}\\ \vspace{-.5cm}
\begin{enumerate}
	\item Let $(B,A)$ be a pair. Then, $	AH^2(\CC^m)\subset \HH(B)$.
	\item If the pair $(B,A)$ is special and $G=(Id-B)^{-1}A$, then $	AH^2(\CC^m)$ is dense in $\HH(B)$ if and only if $G^2$ is rigid.
	\item If $AH^2(\CC^m)$ is dense in $\HH(B)$ (relatively to $\|\cdot\|_{\HH(B)}$), then $(B,A)$ is special.
\end{enumerate}
\end{prop}
Consequently, if $AH^2(\CC^m)$ is dense in $\HH(B)$, then $G^2$ is rigid.   

\begin{proof}
\begin{enumerate}
	\item Using the property of Toeplitz operators, we obtain easily:
	\[
	\forall f\in H^2(\CC^m,\mu), \ T_{Id-B}T_{G^*}T_{G^{*-1}G}f=T_{Id-B} Gf=(Id-B)Gf=Af.
	\]
	So, $T_{Id-B}T_{G^*}$ sends the range of $T_{G^{*-1}G}$ in $AH^2(\CC^m)$. But, if $(B,A)$ is  special pair, then the range of $T_{Id-B}T_{G^*}$ is $\HH(B)$. So, $AH^2(\CC^m)$ is contained in all $\HH(B)$.
	\item If $(B,A)$ is special, then $T_{Id-B}T_{G^*}$ is a surjective isometry from $H^2(\CC^m)$ to $\HH(B)$.   Because $A$ is outer,  $AH^2(\CC^m)$   is dense in $H^2(\CC^m)$, but  $AH^2(\CC^m)$   is in $\HH(B)$ so $AH^2(\CC^m)$   is dense in $\HH(B)$.
	Because $T_{Id-B}T_{G^*}T_{G^{*-1}G}=T_A$, and $T_{Id-B}T_{G^*}$ is an isometry, $AH^2(\CC^m)$  is dense is equivalent  to the range of $T_{G^{*-1}G}$ is dense in $H^2(\CC^m)$. This means that the kernel of $T_{G^*G^{-1}}$ is reduced to zero. We conclude by   Proposition~\ref{injectif} that $G^2$ is a rigid function.
	\item We need to prove that the range of $T_{Id-B}T_{G^*}$ is $\HH(B)$. If $AH^2(\CC^m)$  is dense  in $\HH(B)$, then the range of $T_{Id-B}T_{G^*}$ is dense in $\HH(B)$. But $T_{Id-B}T_{G^*}$ is an isometry, so it is all of $\HH(B)$. We conclude that $(B,A)$ is a special pair.	
\end{enumerate}

\end{proof}

In the  scalar case, for a function $b \in H^\infty$ verifying $\log (1-|b|^2)\in L^1$, the polynomials are dense in $\HH(b)$ (\cite{sarason:book} II-4). We verify the   matricial analogue, namely that $\{pu,\ p\in Pol_+,\ u\in \CC^m\}$ is dense in $\HH(B)$. 
 
 Because $T_AT_{A^*} \leq T_{A^*}T_A$, and $T_{A^*}T_A=Id-T_{B^*}T_B$ and $T_BT_{B^*}\leq T_{B^*}T_B$, 
Douglas's Lemma~\ref{douglas} implies that  $\MM(A)\subset \MM(A^*)= \HH(B^*)\subset \HH(B)$.

Let $p$ be a polynomial  and $u$ be a vector in $\CC^m$. The range $T_{A^*}pu$ is of the form $qu$, where $q$ is a polynomial  with the same degree as $p$. Because $\{pu,\ p\in Pol_+,\ u\in \CC^m\}$ is dense in $H^2(\CC^m)$, so it is in $\MM(T_{A^*})=\HH(B^*)$. To complete the proof, we need to show that $\MM(A^*)$ is dense in $\HH(B)$. 

The link between $\HH(B)$ and $\HH(B^*)$ is a corollary of Douglas's Lemma (see \cite{sarason:book} I-8): 
$h$ is in $\HH(B)$ is equivalent to $T_{B^*}h\in \HH(B^*)$.

Because $A$ is outer, $\ker T_{A^*}=\{0\}$, so there exists an unique $h^+$ such that $T_{B^*}h=T_{A^*}h^+$. Moreover, the following  formula (see \cite{sarason:book} IV-1) holds in the matricial case:
\[
\forall h_1,h_2\in \HH(B), \ \ \<h_1,h_2\>_B=\<h_1,h_2\>_2+\<h_1^+,h_2^+\>_2.
\]
Let $h$ be in $\HH(B)$ orthogonal to $\MM(A^*)$ (for the inner product of $\HH(B)$). Then,  for all $n\geq 0,$ $\<h , T_{A^*}S^{*n}h\>_{\HH(B)}=0$, and by unicity, $(T_{A^*}S^{*n}h)^+=T_{A^*}S^{*n}h^+$. 

Now, we can show that $h$ is null. For all $n\geq 0$, we have
\[
\begin{aligned}
0&=\<h,T_{A^*}S^{*n}h\>_{\HH(B)}=\< h,T_{A^*}S^{*n}h\>_2+\< h^+,T_{A^*}S^{*n}h^+\>_2\\
&=\dfrac{1}{2\pi}\int_\TT \left(\<A(e^{it}) h(e^{it}),h(e^{it})\>_{\CC^m}+\<A(e^{it}) h^+(e^{it}),h^+(e^{it})\>_{\CC^m}\right) e^{int}\ dt.
\end{aligned}
\]
The scalar function $\phi:e^{it}\mapsto \<A(e^{it}) h(e^{it}),h(e^{it})\>_{\CC^m}+\<A(e^{it}) h^+(e^{it}),h^+(e^{it})\>_{\CC^m}$ lies in $H_0^1$. The same calculation with $\<T_{A^*}S^{*n}h,h\>_{\HH(B)}=0$ implies that $\overline{\phi}\in H^1_0$, so $\phi$ is constant equal to zero, and we obtain the density of $\{pu,\ p\in Pol_+,\ u\in \CC^m\}$ in $\HH(B)$. 

\begin{cor}
Let $(B,A)$ be a special pair. If $G^2$ is rigid, and $U\in H^\infty(\CC^m\to\CC^m)$ is an inner function of rank $m$, then $(UB,A)$ is special and $((I-UB)^{-1}A)^2$ is rigid.
\end{cor}

A consequence of $((I-UB)^{-1}A)^2$  being rigid  is that $(I-UB)^{-1}A$ is outer. 

\begin{proof}

Because $U$ is inner and $(B,A)$ is a pair, $(BU,A)$ is a pair too. 
 
  Once again,  Douglas's Lemma~\ref{douglas} allows us to show that $\HH(B)\subset \HH(UB)$. We have $T_{UB}=T_UT_B\leq T_B$, so is
$Id-T_B T_{B^*}\leq I-T_{UB}T_{B^*U^*}$ and $\HH(B)\subset \HH(UB)$. 

The polynomials are dense in $\HH(B)$ and in $\HH(UB)$, so $\HH(B)$ is dense in $\HH(UB)$.
Thanks to  Proposition~\ref{prop6}, $G^2$ rigid and $(B,A)$  special implies that $AH^2(\CC^m)$ is dense in $\HH(B)$, so in $\HH(UB)$.  Proposition~\ref{prop6} again, tell us that  the pair $(UB,A)$ is a special. As a consequence, the corresponding function\\ $\left((I-UB)^{-1}A\right)^2$ is rigid.

\end{proof}

\begin{lem}\label{lem3}
Let $(B,A)$ be a pair and let $U\in H^\infty(\CC^m\to\CC^m)$ be an inner function of rank $m$ verifying $B=UB_0$. Let $A':=(I-B_0U)G$. Then $T_{Id-B} T_{G^*}$ maps the range of the operator $T_{G^{*-1}UG}$ onto $UA'H^2(\CC^m)$.
\end{lem}

\begin{proof}
By construction, it is clear that $A'$ is outer in  $H^\infty(\CC^m\to \CC^m)$.
Because $B=UB_0$, we obtain the equality \[(Id-B)UG=(Id-UB_0)UG=U(Id-B_0U)G=UA'.\] So the range of $T_{Id-B} T_{G^*} T_{G^{*-1}UG}$ is $UT_{A'}H^2(\CC^m)$.

\end{proof}

Now, we state the matricial analogue of Hayashi's theorem. 

\begin{thm}
Let $\FF=GK_U$ be a nearly $S^*$-invariant subspace of $H^2(\CC^m)$, where $G\in H^2(\CC^m\to\CC^m)$ is outer and $U\in H^\infty(\CC^m\to\CC^m)$ is inner, verifying $U(0)=0$ and $\rank U=m$. We write \[ A:=(Id-UB_0)G, \ \  A':=(Id-B_0U)G \ \mbox{ and } G_0':=(Id-B_0)A'.\]

Then, $\FF$ is the kernel of a Toeplitz operator if and only if $B=UB_0$, the pair $(B_0,A')$ is special  and $G_0'^2$ is rigid.

\end{thm}

\begin{proof}
Let  $\FF=GK_U$ be the kernel of a Toeplitz operator. Then, Theorem~\ref{theoprincipal} implies that $B=UB_0$, and   Lemma~\ref{casgeneralr=m}  gives that $\FF=\ker T_{G^*U^*G^{-1}}$ and so 
\[
H^2(\CC^m)=\FF\oplus^\perp \overline{T_{G^{*-1}U G} H^2(\CC^m)}.
\]

The section I-10 of \cite{sarason:book}, valid in the matricial case,  gives the following orthogonal decomposition:
\[
\HH(B)=K_U \oplus^{\perp_B} U\HH(B_0).
\]
The operator $T_{Id-B}T_{G^*}$ is an isometry from $H^2(\CC^m)$ to $\HH(B)$. Lemma~\ref{lem3} tells us that it maps the range of $T_{G^{*-1}UG}$ onto $UT_{A'}H^2(\CC^m)$. Moreover, it maps $\FF$ on $K_U$. 
So, we get the following diagram:
\[
\begin{array}{ccccl}
&H^2(\CC^m)=&\FF&\oplus^\perp&\overline{T_{G^{*-1}U G } H^2(\CC^m)}\\
T_{Id-B}T_{G^*} & \downarrow  &\downarrow & & \downarrow \\
&\HH(B)=&K_U&\oplus^{\perp_{\HH(B)}}&U\HH(B_0).\\
\end{array}
\]
It follows that $UA'H^2(\CC^m)$ is dense in $U\HH(B_0)$. On $\HH(B)$, $T_U$ is an isometry, so $A'H^2(\CC^m)$  is dense in  $\HH(B_0)$. We conclude this implication  using  Proposition~\ref{prop6}. The fact that  $A'H^2(\CC^m)$ is dense in  $\HH(B_0)$ implies that $(B_0,A')$ is special and $G_0'^2$ is rigid.

Conversely, we can reverse the reasoning. The pair $(B_0,A')$ is special  and $G_0'^2$ is rigid imply that $A'H^2(\CC^m)$  is dense in  $\HH(B_0)$ and so $UA'H^2(\CC^m)$ is dense in $U\HH(B_0)$. The diagram holds to be true, so 
\[
T_{Id-B}T_{G^*} \overline{ T_{G^{*-1}UG} H^2(\CC^m) }^{\perp}=T_{Id-B}T_{G^*}\left(U\HH(B_0)\right)^\perp=K_U.
\]
It follows that $\FF=GK_U=\overline{ T_{G^{*-1}UG} H^2(\CC^m)}^{\perp}$ is the kernel of the Toeplitz operator $T_{G^*U^*G^{-1}}$.

\end{proof}

As mentioned in Sarason's article \cite{sarasontoeplitz}, the proof  contains a recipe for constructing a non-trivial proper subspace $\FF\subset H^2(\CC^m)$ which is the kernel of a Toeplitz operator. We repeat the process.

We begin with the particular case $r=m$: Take an outer function $G_0'\in H^2(\CC^m\to \CC^m)$ such that $G_0'^2$ is rigid and with an inner function $U\in H^\infty(\CC^m\to \CC^m)$ vanishing at zero. The pair associated to $G_0'$ is $(B_0,A')$. Let $B=UB_0$, $G=(Id-B_0U)^{-1}A'$ and $A=(Id-UB_0)G$. Thanks to Proposition~\ref{prop6}, the pair $(B,A)$ is special and $G^2$ is rigid. Then $\FF=GK_U$ is a nearly $S^*$-invariant subspace which is the kernel of the Toeplitz operator with symbol  $G^*U^*G^{-1}$. With this construction, $\dim W$ is equal to $m$.

We can adapt the general case $r<m$ from the particular case $r=m$.  With the previous notation, $\FF=GK_U,$ with $G\in H^2(\CC^r\to \CC^m)$ outer, so $\GG=H^2(\GGG)$ where $\GGG$ is a subspace of $\CC^m$ of dimension $r$.   Working in $H^2(\GGG)$ allow us to apply the previous theorem and the unitary matrix $\Theta_0$ to come back in $H^2(\CC^m)$.

\bibliographystyle{alpha}
\bibliography{arXiv_kernelsToeplitz}
\end{document}